\documentclass{amsproc}%
\usepackage{eurosym}
\usepackage{amssymb}
\usepackage{amsfonts}
\usepackage{amsmath}
\usepackage{graphicx}%
\theoremstyle{plain}

\newtheorem{definition}{Definition}
\newtheorem{example}{Example}

\newtheorem{lemma}{Lemma}

\newtheorem{remark}{Remark}

\newtheorem{theorem}{Theorem}
\numberwithin{equation}{section}
\begin{document}
\title[A new representation of Links: Butterflies]{A new representation of Links: Butterflies}
\author{H. M. Hilden}
\address{University of Hawaii at Honolulu}
\email{mike@math.hawaii.edu}
\author{J. M.\ Montesinos}
\curraddr{Universidad Complutense de Madrid}
\email{montesin@mat.ucm.es}
\author{D. M. Tejada }
\curraddr{Universidad Nacional de Colombia, Sede Medell\'{\i}n}
\email{dtejada@unalmed.edu.co}
\author{M. M. Toro}
\curraddr{Universidad Nacional de Colombia, Sede Medell\'{\i}n}
\email{mmtoro@unalmed.edu.co}
\thanks{}
\date{February 2012}
\subjclass[2000]{Primary 57M25, 57M27 }
\keywords{Links, bridge presentation, bridge number, butterfly, butterfly number}

\begin{abstract}
With the idea of an eventual classification of 3-bridge links,\ we define a
very nice class of $3$-balls (called butterflies) with faces identified by
pairs, such that the identification space is $S^{3},$ and the image of a
prefered set of edges is a link. Several examples are given. We prove that
every link can be represented in this way (butterfly representation). We
define the butterfly number of a link, and we show that the butterfly number
and the bridge number of a link coincide. This is done by defining a move on
the butterfly diagram. We give an example of two different butterflies with
minimal butterfly number representing the knot $8_{20}.$ This raises the
problem of finding a set of moves on a butterfly diagram connecting diagrams
representing the same link. This is left as an open problem.

\end{abstract}
\maketitle

\section{\label{Intro}Introduction}

The beautiful classification of $2$-bridge links by rational numbers has not
yet been generalized to 3-bridge links. One of the goals of this paper is to
introduce a tool that eventually might lead to a generalization of this classification.

It is well known \cite{seifert} that every closed, orientable $3$-manifold can
be obtained by pasting pairs of faces of a polygonization of the boundary
$S^{2}$ of a closed $3$-cell $\mathbf{B}^{3}$.

Thurston's construction of the borromean rings, \cite{Th1} and \cite{Th2}, is a nice
example that we generalize for all links in this paper, Fig. \ref{borro}. In
this example we notice that the cube is actually a closed $3$-cell
$\mathbf{B}^{3},$ with twelve faces on its boundary that are identified by
reflections along some axes (double arrows). Moreover, pasting the faces of
the cube we obtain $S^{3}$ and the set of axes become the borromean rings.

These reflections resemble the way a butterfly closes its wings, and we will
say that the borromean rings have a $6$-butterfly representation, and the six
faces of the real cube are the six butterflies involved.

\begin{figure}
[h]
\begin{center}
\includegraphics[
height=1.8325in,
width=4.1926in
]%
{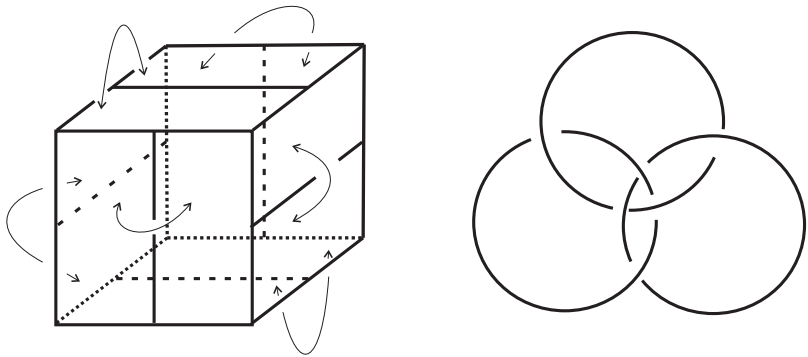}%
\caption{Borromean rings.}
\label{borro}
\end{center}
\end{figure}

Similarly to the borromean rings, the $2$-bridge knots or rational links $p/q$
can be obtained by pasting the northern and southern hemispheres of $S^{2}$
with themselves by reflections through half meridians separated apart $2\pi
q/p.$ For instance, Fig. \ref{fig16} depicts this construction for $p/q=3/1$,
the trefoil knot. As in Thurston's example, $S^{3}$ is obtained by pasting the faces. We say that the rational
link $p/q$ has a $2$-butterfly representation, and the northern and southern
hemispheres of $S^{2}$ are the two butterflies involved.

This butterfly representation of $p/q$ has two main advantages. First, it is a
pure $2$-dimensional diagram, and secondly, it exhibits directly the rational
number $p/q$ that classifies the knot or link.

With these two properties in mind, we wondered if all knots and links have a
similar structure, allowing two or more butterflies on the boundary $S^{2}$ of
$\mathbf{B}^{3}.$ One such structure with three butterflies is depicted in
Fig. \ref{fig2}b.

It turns out that every knot or link admits such a representation. We prove
this fact here. In Sections \ref{ButLink} and \ref{LinkBut} we give algorithms
to pass back and forth from a link to a butterfly representation of it.

We define accordingly the butterfly number of a knot or link and we prove that
it coincides with its bridge number (Section \ref{number}). To obtain this
last result we need to reduce the number of butterflies of a particular
butterfly representation of a link. This involves the definition of a move
that does precisely this. See Section \ref{Move}.

As each $m$-bridge link diagram has an $m$-butterfly representation, a natural
question arises: \textit{Is it possible to associate a set of rational numbers
to describe this butterfly}? In the case $m=3$ this assignation can, in fact,
be made \cite{HMTT4}, where a triple of rational numbers is associated to each
3-butterfly. In this paper we show some examples of $3$-butterflies and its
corresponding set of rational numbers.

We give many examples and in particular two different $3$-butterfly
representations of the same knot $8_{20}.$ This raises the problem of relating
$3$-butterfly representations by a set of potential moves. This is left as an
open problem. Using the concept of $3$-butterfly, we hope to obtain a
classification of $3$-bridge links, similar to the Schubert classification of
$2$-bridge links.

In Section \ref{Def} we present a technical definition of an $m$-butterfly
even though in the rest of the paper, for simplicity, we speak more
intuitively about $m$-butterflies.

In the last decade, Kauffman \cite{Ka} has been developing the theory of
virtual knots. This theory has several applications. The technical definition
of an $m$-butterfly is used intensively in \cite{HMTT5} where we prove that
any virtual knot also admits a representation by a generalized $(n,g)$%
-butterfly, that is a handlebody of genus $g$ with $2n$ faces on its boundary
that are identified by reflections along some axis.

As we have remarked above, pasting the faces of an $m$-butterfly gives the $3
$-sphere $S^{3}$. Section \ref{S3} is devoted to showing this fact. In
general, this result is not true for generalized $(n,g)$-butterflies that
represent virtual knots.

\section{\label{Def}Butterflies: Definitions and Examples}

Intuitively, an $m$-butterfly is a $3$-ball $\mathbf{B}^{3}$ with $m$ $>0$
polygonal faces on its boundary $S^{2}=\partial\mathbf{B}^{3},$ such that each
face $P$\ is subdivided by an arc $t_{P}$\ in two subfaces (that have the same
number of vertices) that are identified by a ''reflection'' along this arc
$t_{P}.$

In order to formalize this concept,\ we give some technical definitions.

Let $R$ be a connected graph embedded in $S^{2}=\partial\mathbf{B}^{3},$ where
$\mathbf{B}^{3}$ is a closed $3$-cell, so that $S^{2}-R$ is a disjoint union
of open 2-cells. For our purposes we assume that $\mathbf{B}^{3}$ is the half
ball $x^{2}+y^{2}+z^{2}\leq r^{2};z\leq0,$ and that the graph $R$ and later
the graph $R\cup T$, when $T$ has been defined, is contained in the planar
part of $\mathbf{B}^{3},\mathbb{R}^{2}\times\{0\}.$ The edges in $R$ and in
$T$ are simple arcs. However, by \cite{F1}, for any such graph $R\cup T$ there
is an autohomeomorphism of $S^{2}$ such that the images of the edges are
straight planar line segments. We shall assume, in the proofs of theorems that
follow, but not in the drawn figures, that the edges of $R\cup T$ are straight
planar line segments. We denote each open $2$-cell generically by $P.$ We
would like to parameterize each $2$-cell $P$.

For any $n\in\mathbb{N}$, let $P_{2n}$ be the regular polygon that is the
closed convex hull of the $2n^{th}\ $roots of unity. We define a
parameterization of $P$ to be a function $f$ from $P_{2n}$ to the closure
$\overline{P}$ of $P,$ with the following properties:

a) The restriction of $f$ to interior $P_{2n}$ is a homeomorphism from
interior $P_{2n}$ to $P.$

b) The restriction of $f$ to an edge of $P_{2n}$ is a piecewise linear
homeomorphism from that edge to an edge in the graph $R.$

c) $f$ as a map from the edges of $\partial P_{2n}$ to the edges of $\partial
P$ is at most 2 to 1.

The existence of a parameterization of $P$ places restrictions on $P$ and on
$R.$ We will assume that $R$ is such that each $P$ has a parameterization
$f:P_{2n}\rightarrow\overline{P},$ and we fix a parameterization $f_{P}$ for
each $P.$

Complex conjugation, $z\rightarrow\overline{z},$ restricted to $P_{2n}$ or to
boundary of $P_{2n}$ defines an involution and an equivalence relation on the
edges and vertices of $P_{2n}$, and this in turn, induces an equivalence
relation on the edges and vertices of $\overline{P},$ and on the points of $P
$ as well. That is to say for $A$ and $B$ points of $\overline{P}$, $A\sim B$
if $f_{P}^{-1}\left(  A\right)  =f_{P}^{-1}\left(  B\right)  $ or $f_{P}%
^{-1}\left(  A\right)  =\overline{f_{P}^{-1}\left(  B\right)  },$ where
$\overline{f_{P}^{-1}\left(  B\right)  }=\left\{  \overline{z}/z\in f_{P}%
^{-1}\left(  B\right)  \right\}  .$

The equivalence relation on the edges and vertices of each $\overline{P}$
induces an equivalence relation on the graph $R.$ That is $x\simeq y$ if and
only if there exists a finite sequence $x=x_{1},\cdots,x_{l}=y$ with
$x_{i}\sim x_{i+1}$ for $i=1,\cdots,l-1.$ Equivalence classes of points of $P
$ contain two points except for those points in $f\left(  \left[  -1,1\right]
\right)  $ where there is only one point. Note that if $x$ is a vertex of $R,$
its complete class under the equivalence relation $\simeq$ is composed
entirely of vertices.

Figures \ref{fig1a} and \ref{fig1b} illustrate two different
parameterizations. In Fig. \ref{fig1a} \ we have $f(1)=f(5)$ and $f(2)=f(4);$
and in Fig. \ref{fig1b} we have $f\left(  0\right)  =f\left(  6\right)
,f\left(  1\right)  =f\left(  5\right)  $ and $f\left(  2\right)  =f\left(
4\right)  .$%

\begin{figure}
[ptb]
\begin{center}
\includegraphics[
height=1.701in,
width=4.4952in
]%
{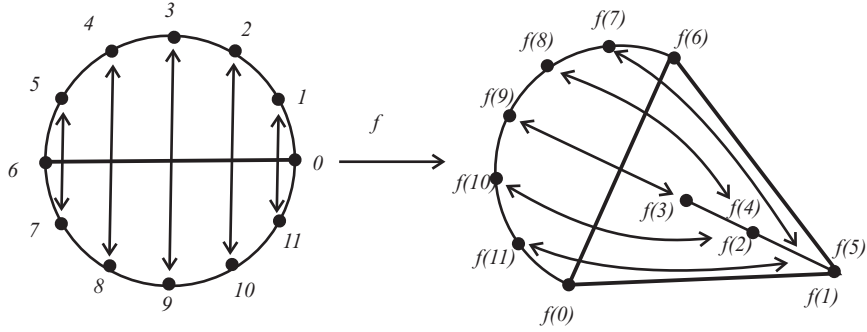}%
\caption{$f$ parameterizes a pair $(P,t)$.}%
\label{fig1a}%
\end{center}
\end{figure}

\begin{figure}
[h]
\begin{center}
\includegraphics[
height=1.9298in,
width=4.5798in
]%
{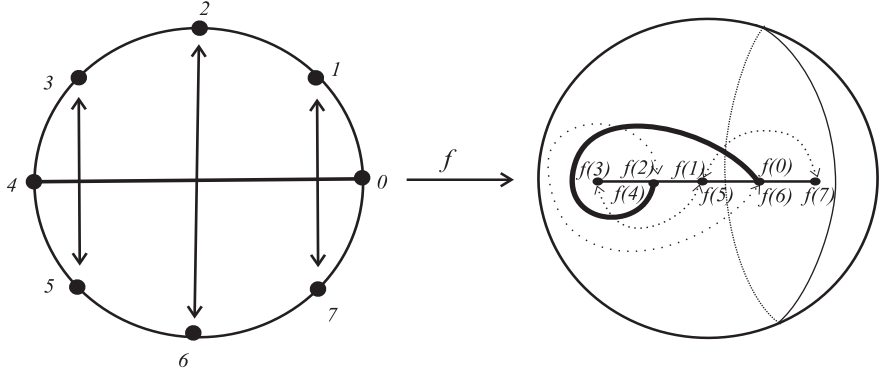}%
\caption{$f$ parametrizes a $1$-butterfly. }%
\label{fig1b}%
\end{center}
\end{figure}

Each $P_{2n}$ contains the line segment $\left[  -1,1\right]  $ which is the
fixed point set of complex conjugation restricted to $P_{2n}.$ The image of
this line segment $f_{p}\left(  \left[  -1,1\right]  \right)  $ is called the
\textit{trunk} $t$. A pair $\left(  P,t\right)  $ will be called a
\textit{butterfly with trunk} $t.$ The \textit{wings} $W$ and $W^{\prime}$ are
just $f_{P}\left(  P_{2n}\cap upper\ half\ plane\right)  $ and $f_{P}\left(
P_{2n}\cap lower\ half\ plane\right)  $ and $W\cap W^{\prime}=t.$ Each time
that we consider a trunk $t$ we are implicitly considering the equivalence
relation described above. We denote by $T$ the collection of all trunks $t$
(over all $P).$ Notice that the boundaries of the $n$ butterflies form a graph
$R$ on $S^{2}=\partial\mathbf{B}^{3}.$ As before, (See \cite{F1}), we can
assume the edges in the graph $R\cup T$ as straight line segments.

Let us denote by $M(R,T)$ the space $\mathbf{B}^{3}/\simeq$ with the topology
of the identification map $p:\mathbf{B}^{3}\rightarrow M(R,T)$.

As in Thurston's example, we would like that the image of $T,p(T),$ became a knot or link. In
order to guarantee this fact, we distinguish three types of vertices on $R.$

A member of $R\cap T$ will be called an $A$\emph{-vertex}. A member of
$p^{-1}\left(  p\left(  v\right)  \right)  ,$ $v\in R\cap T,$ which is not an
$A$-vertex will be called an $E$\emph{-vertex}. A vertex of $R$ which is not
an $A$-vertex nor an $E$-vertex will be called a $B$\emph{-vertex }iff
$p^{-1}\left(  p\left(  v\right)  \right)  $ contains at least one
non-bivalent vertex of $R$.

We do not give an explicit name for those vertices that are neither $A,B$ nor
$E$-vertices. Of course it is possible to construct $3$-balls with polygonal
faces on their boundaries with those kind of vertices but for our purposes (we
want to represent knots or links) it is enough to consider graphs without
them. There are also interesting examples in which there are $E$-vertices that
are not bivalent, as the one shown in Fig. \ref{exemike}, but for our purpose
we do not consider them as $m$-butterflies. In further research we will
consider some generalization of our construction.%

\begin{figure}
[h]
\begin{center}
\includegraphics[
height=1.6164in,
width=2.334in
]%
{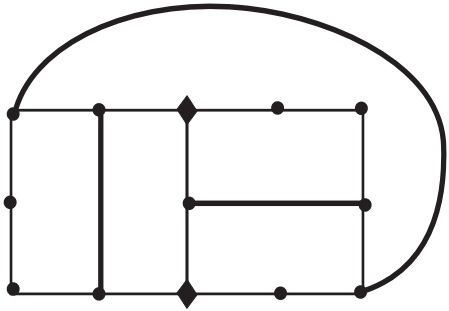}%
\caption{In this poligonization, the vertices marked with $\blacklozenge$ are
trivalent $E$-vertices of $R$. }%
\label{exemike}%
\end{center}
\end{figure}

With these definitions we formalize our intuitive definition of $m$-butterfly,
given at the beginning of this section.

\begin{definition}
For $m\geq1$, an $m$-\textit{butterfly} is a $3$-ball $\mathbf{B}^{3}$ with
$m$ butterflies $(P,t)$ on its boundary $S^{2}=\partial\mathbf{B}^{3},$ such
that (i) the graph $R$ has only $A$-vertices, $E$-vertices and $B$-vertices;
(ii) the $A$- and $E$-vertices are bivalent in $R,$ and (iii) $T$ has $m$ components.
\end{definition}

Moreover, an $m$-butterfly can be represented by a planar graph (or by an
$m$-\textit{butterfly diagram}), denoted by a pair $\left(  R,T\right)  ,$
such that conditions (i), (ii), and (iii) are satisfied. The $m$-butterfly
represented by the diagram $(R,T)$ is also denoted by $(R,T).$

\begin{example}
Figure \ref{fig2} depicts three different butterfly diagrams. Fig \ref{fig2}b
represents a $3$-butterfly. The full equivalence class of the two trivalent
vertices $0$ and $\infty$ on it are $B$-vertices. Fig. \ref{fig2}a shows a
$2$-butterfly that has only $A$ or $E$-vertices, while the $1$-butterfly given
in \ref{fig2}c has only two $A$-vertices and three $B$-vertices.
\end{example}

In the examples of Fig. \ref{fig2} we will assume that $\mathbf{B}^{3}$ is the
closed $3$-cell that lies over the paper in $\mathbb{R}^{3}+\infty$. The
members of $T$ will be displayed as thick lines. The $B$-vertices are depicted
by *. See \ref{fig2}b and c. The other vertices of the diagram are either
boundaries of members of $T$ ($A$-vertices) or $E$-vertices.%

\begin{figure}
[h]
\begin{center}
\includegraphics[
height=1.5904in,
width=5.0713in
]%
{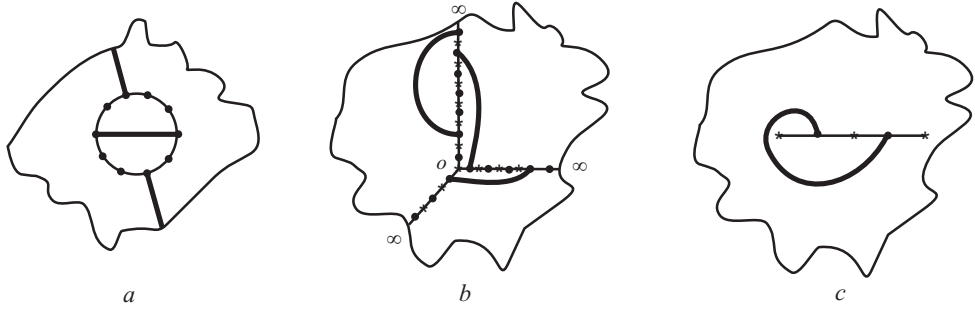}%
\caption{Representing butterflies with planar graphs.}%
\label{fig2}%
\end{center}
\end{figure}

\section{\label{S3}The Quotient Space $M(R,T)$ is $S^{3}$}

In this section we are going to prove that under our definitions, the space
$M(R,T)$ is $S^{3}$ and that the image of $T$ under the identification map $p$
is a knot (or link). So we are sure to obtain a knot (or link) inside $S^{3}$
when we make the identifications by the equivalence relation$.$

\begin{theorem}
\label{teos3}For any $m$-butterfly $(R,T),$ the space $M\left(  R,T\right)  $
is homeomorphic to $S^{3}$ and $p\left(  T\right)  $ is a knot or a link,
where $p:\mathbf{B}^{3}\rightarrow M(R,T)$ is the identification map.
\end{theorem}

\begin{proof}
Set\thinspace$M=M\left(  R,T\right)  \ $for shortness. Let $R^{\ast}=p(R)$,
$T^{\ast}=p(T)$ and $V^{\ast}=p(V)$, where $V$ is the set of vertices of $R. $
Let $U(V^{\ast})$ be a regular neighbourhood of $V^{\ast}$ in the space
$M=M\left(  R,T\right)  .$ Then $U(V^{\ast})$ is a disjoint union of regular
neighbourhoods (we choose $U(V^{\ast})$ as small as we need) of the vertices
of $V^{\ast}$. Let $v^{\ast}\in V^{\ast}$ be one of these vertices. Of course
any regular neighbourhood of $v^{\ast}$ is the cone over an orientable surface
$\Sigma_{v^{\ast}}.$

\textbf{Claim 1: }The surface $\Sigma_{v^{\ast}}$ is connected.

\textbf{Proof: }Consider the subset $p^{-1}\left(  v^{\ast}\right)  $ of the
set $V.$ Let $v\in p^{-1}\left(  v^{\ast}\right)  .$ A regular neighbourhood
of $v$ in $\mathbf{B}^{3}$ is a cone from $v$ over a $2$-disk $\Delta_{v}$
properly embedded in $\mathbf{B}^{3}$. Denote this cone by $C\left(
v,\Delta_{v}\right)  .$ It is possible to select the regular neighbourhood of
members of $p^{-1}\left(  v^{\ast}\right)  $ so that
\[
\Sigma_{v^{\ast}}=
{\textstyle\bigcup_{v\in p^{-1}\left(  v^{\ast}\right)  }}p\left(  \Delta_{v}\right).
\]
Now, if $v_{1},v_{2}\in p^{-1}\left(  v^{\ast}\right)  $ then $v_{1}\simeq
v_{2},$ so there exits a finite sequence of vertices of $p^{-1}\left(
v^{\ast}\right)  $ say $u_{1}=v_{1},u_{2},\cdots,u_{k}=v_{2}$ such that
$u_{i}\sim u_{i+1},i=1,\cdots,k-1$. If we assume that $u_{i},u_{i+1}$ belong
to some $\overline{P},$ where $(P,t)$ is the corresponding butterfly, then the
boundary of $\Delta_{u_{i}}\cap P$ and $\Delta_{u_{i+1}}\cap P$ are also
identified and it follows that $p\left(  \Delta_{u_{i}}\right)  \cup p\left(
\Delta_{u_{i+1}}\right)  $ is a connected set. From this, the claim follows easily.

We continue with the proof of the theorem. The closure of $M\smallsetminus
U\left(  V^{\ast}\right)  $ is clearly a compact, connected $3$-manifold
$M^{\ast}$ with boundary $\partial M^{\ast}=
{\textstyle\bigcup_{v^{\ast}\in V^{\ast}}}\Sigma_{v^{\ast}}.$ 
The closure in $M^{\ast}$ of the set $R^{\ast
}\smallsetminus U\left(  V^{\ast}\right)  $ (resp. $T^{\ast}\smallsetminus
U\left(  V^{\ast}\right)  $) is a set of disjoint, properly embedded arcs in
$M^{\ast}$ that will be denoted by $R^{\ast\ast}$ (resp. $T^{\ast\ast}$).

Now drill from $M^{\ast}$ a regular neighbourhood $U(R^{\ast\ast})\cup
U\left(  T^{\ast\ast}\right)  $ of $R^{\ast\ast}\cup T^{\ast\ast}$ and take
the closure $M^{\ast\ast}$ of the result. Then $M^{\ast\ast}$ is the image
under $p$ of the ball $C=\mathbf{B}^{3}\smallsetminus U\left(  R\cup T\right)
,$ where $U\left(  R\cup T\right)  $ is a suitable regular neighbourhood of
$R\cup T.$ The set $\partial\mathbf{B}^{3}\smallsetminus U\left(  R\cup
T\right)  $ is a system $\left\{  \check{W}_{1},\check{W}_{1}^{^{\prime}%
},\cdots,\check{W}_{m},\check{W}_{m}^{^{\prime}}\right\}  $ of $2m$ disks in
$\partial C.$ Here $\check{W}_{i},\check{W}_{i}^{^{\prime}} $ are contained in
the wings $W_{i},W_{i}^{^{\prime}}$ of the butterfly $P_{i}\ $with trunk
$t_{i}$ and $p$ identifies $\check{W}_{i},\check{W}_{i}^{^{\prime}}$. Thus
$M^{\ast\ast}$ is a handlebody. Therefore $M=M^{\ast\ast}\cup U\left(
T^{\ast\ast}\right)  \cup U(R^{\ast\ast})\cup U(V^{\ast}).$ The set $U\left(
T^{\ast\ast}\right)  $ is a set of $m $ $2$-handles that are attached to the
handlebody $M^{\ast\ast}.$ The attaching spheres for these $2$-handles are
meridians $\mu_{1},\cdots,\mu_{m}$ of $p\left(  t_{1}\right)  ,\cdots,p\left(
t_{m}\right)  $. Then $\mu_{i}$ cuts $p\left(  \check{W}_{i}\right)  =p\left(
\check{W}_{i}^{^{\prime}}\right)  $ transversely in just one point. Therefore
$M^{\ast\ast}\cup U\left(  T^{\ast\ast}\right)  $ is a $3$-ball $C^{3}$. Thus
\[
M=C^{3}\cup U(R^{\ast\ast})\cup U(V^{\ast}).
\]
Since $U(R^{\ast\ast})$ are $2$-handles attached to $C^{3}$ it follows that
$C^{3}\cup U(R^{\ast\ast})$ is a punctured $3$-ball. Since the boundary of
$C^{3}\cup U(R^{\ast\ast})$ and $U(V^{\ast})$ coincide, it follows that
$\partial U(V^{\ast})$ is a disjoint union of spheres. From the above claim,
it follows that $U(V^{\ast})$ is a disjoint union of cones over spheres. That
is, $U(V^{\ast})$ is a disjoint union of balls. Then $M$ is homeomorphic to
$S^{3}.$

To prove that $p\left(  T\right)  $ is a knot or a link, it is enough to show
that $p^{-1}(p(v)),$ for every $A$-vertex $v$, contains exactly two
$A$-vertices. To prove this we construct the following graph $\Gamma.$

Assume that the $3$-cell $\mathbf{B}^{3}$ is the upper half space
$\mathbb{R}_{+}^{3}$ of $\mathbb{R}^{3}+\infty$, and that the graph $R$ lies
in its boundary $\mathbb{R}^{2}\times\left\{  0\right\}  .$

Let $(P,t)\ $be a butterfly of $(R,T)$ and let $f_{P}:P_{2k}\rightarrow
\overline{P}$ be its fixed parameterization. Let $w_{1},w_{2},\cdots,w_{2r}$
be the vertices of $P_{2k}$ and let $v_{j}=f_{P}(w_{j})$ be the vertices of
$\partial P$. For a vertex $w_{j}=\cos(k\pi/r)\pm i\sin(k\pi/r),$
$k=1,2,...,r-1,$ let $L(w_{j})$ be the open vertical line segment $(\cos
(k\pi/r)+i\sin(k\pi/r),\cos(k\pi/r)-i\sin(k\pi/r)).$ For each $A$-vertex in
$\partial P$ not in $\partial t\ $and each $E$-vertex $v_{j}=f_{P}(w_{j})$ in
$\partial P,$ take the arc $Q_{v_{j}}=f_{p}(L(w_{j})). $

Denote by $\Gamma$ the union of all possible $Q_{v}$'s for any $v\in R$ that
is an $A$- or $E$-vertex.

\noindent\textbf{Claim 2:} $\Gamma$ is a disjoint union of arcs bounded by $A$-vertices.

\textbf{Proof: \ }(1) Noting that if $v$ is an $A$-vertex then any other
vertex, related to it, is an $A$- or $E$-vertex and it follows that the
vertices of $\Gamma$ are all $A$- or $E$-vertices.

(2) Since by definition the $A$-vertices are bivalent in $R$ and they are end
points of some trunk it follows that they are monovalent vertices of $\Gamma.$

(3) Since by definition the $E$-vertices are bivalent in $R$ and they are not
end points of a trunk it follows that they are bivalent vertices of $\Gamma.$

Thus, each component of the graph $\Gamma$ is linear and it is bounded by two
$A$ vertices.

To finish the proof of the theorem we observe that if $\Gamma_{0}$ is a
component of $\Gamma,$ the set of vertices of $\Gamma_{0}$ form a complete
equivalence class under $\simeq.$ Therefore, $p^{-1}\left(  p\left(  v\right)
\right)  $ for every $A$-vertex $v$ contains exactly two $A$-vertices. Hence
the graph $p\left(  T\right)  $ is a knot or a link.
\end{proof}

\begin{definition}
The knot or link $p(T)$ defined by the $m$-butterfly $\left(  R,T\right)  $
will be denoted by $L(R,T)$, and we say that $L(R,T)$ has the butterfly
representation $\left(  R,T\right)  $ with butterfly number $m,$ or that the
$m $-butterfly diagram $\left(  R,T\right)  $ represents $L(R,T)$.
\end{definition}

\section{\label{ButLink}From the Butterfly to the Link}

In this section we show how to construct the link $L(R,T)$ from an
$m$-butterfly $(R,T).$

Recall, \cite{Mu}, that a regular diagram $D_{L}$ of a link $L$ is an
$m$-\textit{bridge diagram} for the link $L$ if we can divide up $D_{L}$ into
two sets of polygonal curves $O=\left\{  o_{1},o_{2},\cdots,o_{m}\right\}  $
and $U=\left\{  u_{1},u_{2},\cdots,u_{m}\right\}  $ $(m>0)$ such that:

i. $D_{L}=o_{1}\cup o_{2}\cup\cdots\cup o_{m}\cup u_{1}\cup u_{2}\cup
\cdots\cup u_{m}$,

ii. $o_{1},o_{2},\cdots,o_{m}$ are mutually disjoint simple curves,

iii. $u_{1},u_{2},\cdots,u_{m}$ are mutually disjoint simple curves,

iv. At the crossing points of $D_{L},$ $o_{1},o_{2},\cdots,o_{m}$ are segments
that pass \textit{over }at least one crossing\textit{\ }point, while
$u_{1},u_{2},\cdots,u_{m}$ are segments that pass \textit{under }at least one
crossing point\textit{. }

The arcs $o_{1},o_{2},\cdots,o_{m}$ are called \textit{bridges or overarcs}.
We use the notation $D_{L}=\left(  O,U\right)  $ when we want to describe
explicitly the bridge presentation of the link $L.$

Note that, by condition iv., there are link diagrams that are not bridge
diagrams. For instance, a simple closed curve is not a \textit{bridge diagram}
for the trivial knot. In this paper, we follow \cite{ChLi} and we differ from
\cite{NeOk}, where it is considered the trivial knot with no crossing as
having an $m$-bridge diagram, for all $m\in\mathbb{N}.$ When a link $L$ has
unknotted components, we need to take some care about them, in order to obtain
an $m$-bridge diagram of $L$ because no component can be expresed as a union
of only $o's$ or $u's$.
 Actually, we have to make at least one kink to the trivial knot to obtain
a bridge diagram for it.

\begin{definition}
Given a link $L,$ the\textsl{\ }\emph{bridge number of} $L$ is the minimum
number $m\,$among of all possible $m$-bridge diagrams of the link $L.$ It is
denoted by $b(L).$
\end{definition}

For example, the trivial knot has bridge number $1$ (see Fig. \ref{fig5}c).

\begin{lemma}
\label{diagram}Given a link $L$, there exists an $m$-bridge diagram $D_{L}$
for $L,$ such that $D_{L}\ $is connected and has no closed curves.
\end{lemma}

\begin{proof}
If the diagram has a closed circle that splits or if it is not connected,
apply the moves shown in Figures \ref{fig19a} and \ref{fig19b}.
\begin{figure}
[h]
\begin{center}
\includegraphics[
height=1.3351in,
width=3.9707in
]
{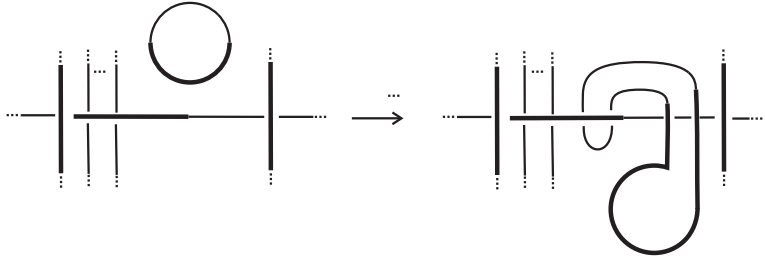}%
\caption{Eliminating closed curves.}%
\label{fig19a}%
\end{center}
\end{figure}

\begin{figure}
[ptbh]
\begin{center}
\includegraphics[
height=1.1149in,
width=3.88in
]
{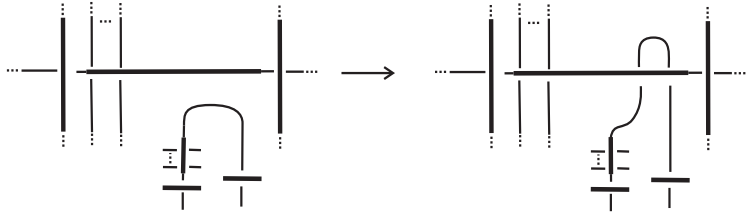}
\caption{Connecting the diagram.}
\label{fig19b}
\end{center}
\end{figure}

\end{proof}

Now, given an $m$-butterfly diagram $\left(  R,T\right)  $ we will describe an
algorithm (the \textit{butterfly-link algorithm}) to construct the link
$L=L(R,T)$. Moreover, we will produce an $m$-bridge diagram for the link
$L\left(  R,T\right)  $.

First of all, consider the following link $K^{\ast}$ of $\mathbb{R}_{+}^{3}$.%

\[
K^{\ast}=\left(  \Gamma\times\left\{  1/2\right\}  \right)  \cup\left(
T\times\left\{  1\right\}  \right)  \cup\left(  \partial T\times\left[
1/2,1\right]  \right)  ,
\]
where $\Gamma$ is the graph defined in the proof of Theorem \ref{teos3}. By
the second claim in the proof of Theorem \ref{teos3}, $\Gamma\times\left\{
1/2\right\}  $ is a disjoint union of arcs lying in $\mathbb{R}^{2}%
\times\left\{  1/2\right\}  .$ Therefore $(\Gamma\times\left\{  1/2\right\}
,T\times\left\{  1\right\}  \cup\left(  \partial T\times\left[  1/2,1\right]
\right)  )$ is an $m$-bridge presentation of the knot (or link) $K^{\ast}$.
This proves the second part of Theorem \ref{teobridge}.

In Fig. \ref{fig3} we illustrate a portion of $K^{\ast}.$ On plane
$\mathbb{R}^{2}\times\{0\}$ we see a component of $\Gamma,\Gamma_{1},$ that is
bounded by two components of $T$ (denoted generically by $T),$ whose
intersection with that $\Gamma_{1}$ is composed of two $A$-vertices (denoted
generically by $A$) and that passes through two $E$-vertices (denoted by $E$).
The points $f,g$ and $h$ are intersections of some components of $T$ with
$\Gamma_{1}$ (we do not depict those components but they are transversal to
$\Gamma_{1}$)$.$%

\begin{figure}
[h]
\begin{center}
\includegraphics[
height=2.0815in,
width=2.6593in
]%
{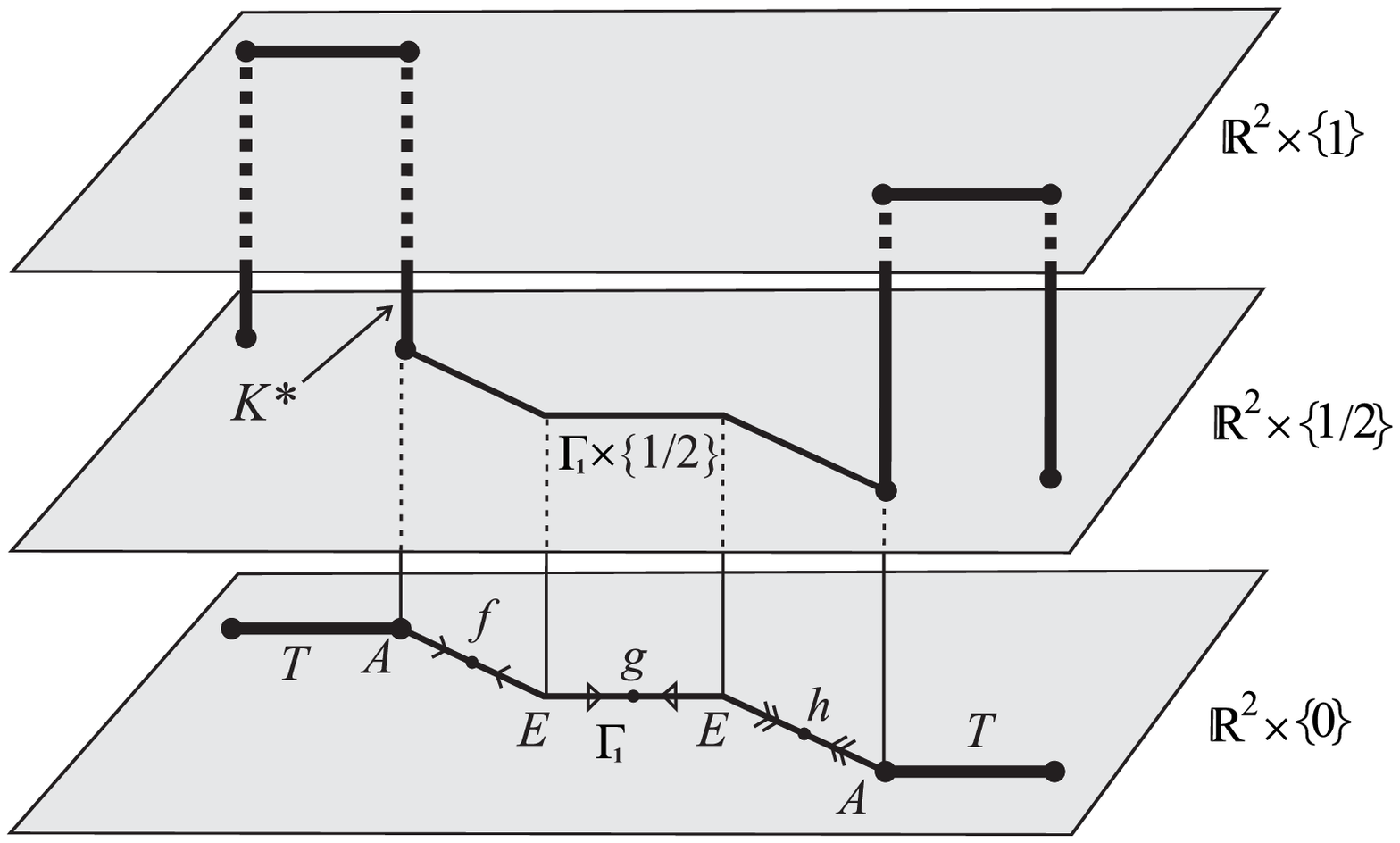}%
\caption{$K^{\ast}=\Gamma\times\left\{  1/2\right\}  \cup T\times\left\{
1\right\}  \cup\partial T\times\left[  1/2,1\right]  $}
\label{fig3}
\end{center}
\end{figure}

\begin{theorem}
\label{teobridge} Given an $m$-butterfly diagram $(R,T)$ the link $L(R,T)$ is
isotopic to $K^{\ast}$. Moreover $(\Gamma\times\left\{  1/2\right\}
,T\times\left\{  1\right\}  \cup\left(  \partial T\times\left[  1/2,1\right]
\right)  )$ is an $m$-bridge presentation of $L(R,T)$.
\end{theorem}

\begin{proof}
Consider a component $\Gamma_{1}$ of $\Gamma.$ It is linear and bounded by two
$A$-vertices. Call $\partial\Gamma_{1}$ the set of these two $A$-vertices. \ 

Consider the subset $\Gamma_{1}\times\left[  0,1/2\right]  $ of$\ \mathbb{\ R}%
_{+}^{3}$. Then $p\left(  \Gamma_{1}\times\left[  0,1/2\right]  \right)  $ is
a cone $C\left(  w,p\left(  \Gamma_{1}\times\left\{  1/2\right\}  \right)
\right)  $ from the point $w=p\left(  \Gamma_{1}^{\left(  0\right)  }%
\times\left\{  0\right\}  \right)  $ over $p\left(  \Gamma_{1}\times\left\{
1/2\right\}  \right)  $ (compare Figures \ref{fig3} and \ref{fig4}) where
$\Gamma_{1}^{\left(  0\right)  }$ is the set of vertices of $\Gamma_{1}.$ We
push $p\left(  \Gamma_{1}\times\left\{  1/2\right\}  \right)  $ along the cone
$C\left(  w,p\left(  \partial\Gamma_{1}\times\left\{  1/2\right\}  \right)
\right)  .$ This we do, as shown in Fig. \ref{fig4}, by an isotopy $H_{i}$
whose final image is just $p\left(  \partial\Gamma_{1}\times\left[
0,1/2\right]  \right)  .$
\begin{figure}
[h]
\begin{center}
\includegraphics[
height=1.9426in,
width=2.1357in
]
{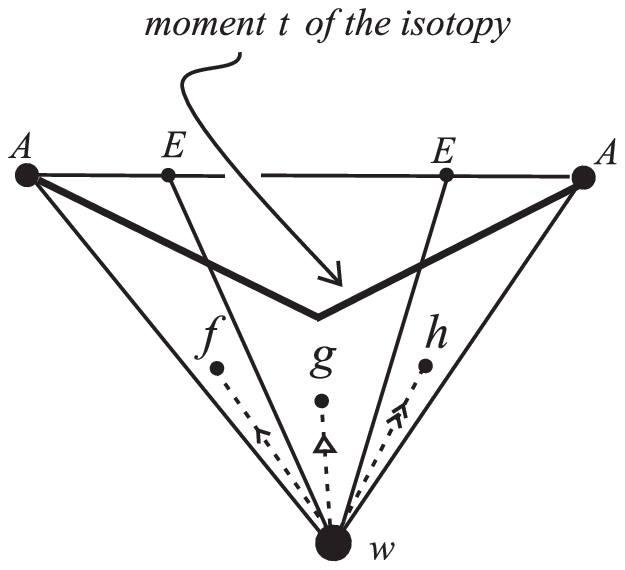}
\caption{Isotopy}
\label{fig4}
\end{center}
\end{figure}

Combining these isotopies $H_{i}$ for all components $\Gamma_{i}$ of $\Gamma$
we obtain an isotopy $H$ sending $K^{\ast}$ onto the set
\[
p\left(  \left(  T\times\left\{  1\right\}  \right)  \cup\left(  \partial
T\times\left[  0,1\right]  \right)  \right)  .
\]
But there is certainly an isotopy $H^{\prime}$ sending $p\left(  \left(
T\times\left\{  1\right\}  \right)  \cup\left(  \partial T\times\left[
0,1\right]  \right)  \right)  $ onto $p\left(  T\times\left\{  0\right\}
\right)  =K. $ This finishes the first part of the proof.
\end{proof}

\textbf{Algorithm (Butterfly-Link algorithm).}

Finally we have:

\begin{itemize}
\item Start with an $m$-butterfly diagram on the plane $\mathbb{R}^{2}%
\times\{0\}$. We want to construct the link $L(R,T)$.

\item Construct the graph $\Gamma\subset$ $\mathbb{R}^{2}\times\{0\}$ as in
the proof of the Theorem \ref{teos3}. See the dotted lines in Fig. \ref{fig5}.

\item Then the link $L(R,T)$ is $\left(  \Gamma\times\left\{  0\right\}
\right)  \cup\left(  T\times\left\{  1\right\}  \right)  \cup\left(  \partial
T\times\left[  0,1\right]  \right)  $.

\item And $(\Gamma\times\left\{  0\right\}  ,T\times\left\{  1\right\}
\cup\left(  \partial T\times\left[  0,1\right]  \right)  )$ is an $m$-bridge
diagram of $L(R,T)$.
\end{itemize}

\begin{example}
Applying the butterfly-link algorithm found in the proof of Theorem
\ref{teobridge} to the three butterfly diagrams of Fig. \ref{fig2} we obtain
the knots of Fig. \ref{fig5}. The knot of Fig. \ref{fig5}a is the knot
$4_{1},$ the knot of Fig. \ref{fig5}b is the knot $8_{20}$ and the knot in
\ref{fig5}c is the trivial knot.
\end{example}

\begin{figure}
[h]
\begin{center}
\includegraphics[
height=1.7513in,
width=5.1102in
]
{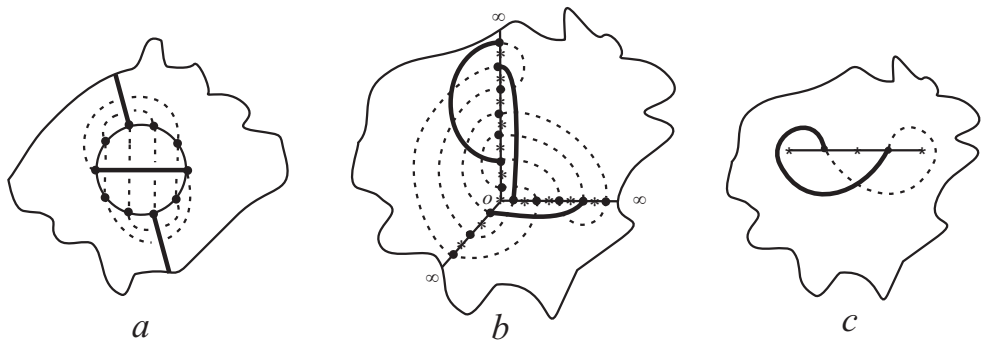}%
\caption{Examples of knots produced by the butterfly-link algorithm.}%
\label{fig5}
\end{center}
\end{figure}

\section{\label{LinkBut}From Links to Butterflies}

Now, in the other direction, we explain how to obtain a butterfly from a given link.

\begin{theorem}
\label{linkbut}Every knot or link can be represented by an $m$-butterfly
diagram, for some $m>0.$ Moreover the $m$-butterfly can be chosen with no $E$-vertices.
\end{theorem}

\begin{proof}
Given a link $L$, let $D_{L}$ be an $m$-bridge diagram of $L,$ connected. See
Fig. \ref{Bvertices}. Usually, in the theory of knots, we do not draw the
dotted lines. We assume that they are under the plane $\mathbb{R}^{2}%
\times\{0\}$ and so the diagram can be seen as a finite collection $T=\left\{
t_{1},\cdots,t_{m}\right\}  $ of disjoint arcs (\textit{no closed curves}) in
the plane $\mathbb{R}^{2}\times\{0\}$. Select\textit{\ }a point $B_{i}$ in
each one of the \textit{regions} of the complement of $D_{L}$ in
$\mathbb{R}^{2}\times\{0\}$. For the unbounded component, set $B_{0}=\infty$.%

\begin{figure}
[h]
\begin{center}
\includegraphics[
height=2.0612in,
width=3.7292in
]%
{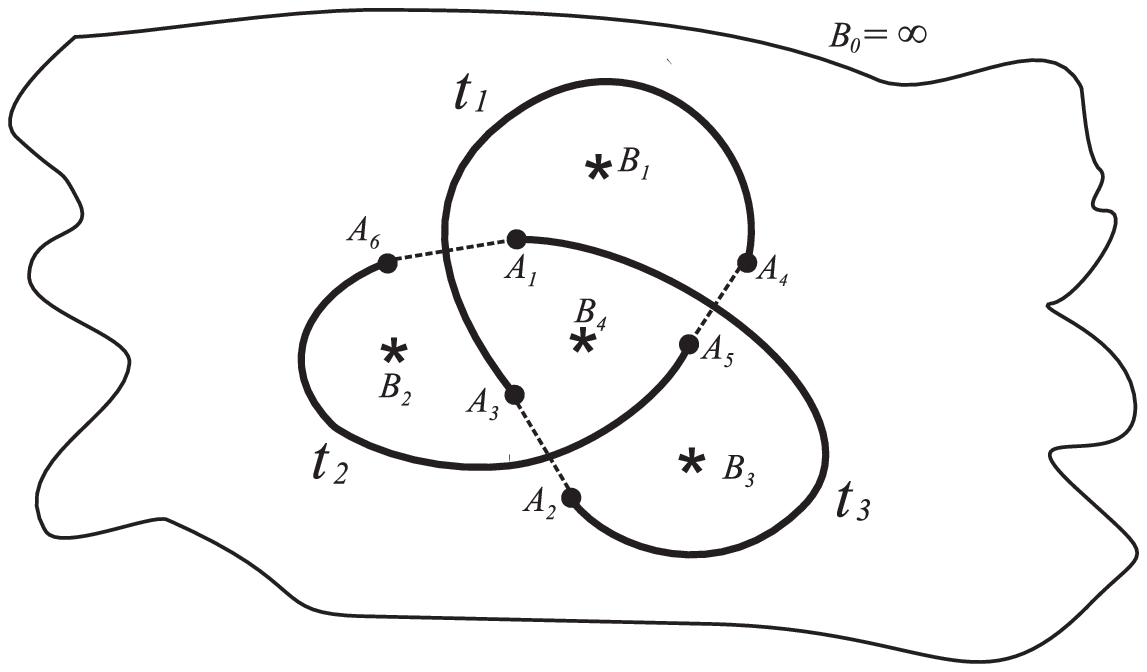}
\caption{Regions of $\mathbb{R}^{2}\backslash D_{L}$.}%
\label{Bvertices}%
\end{center}
\end{figure}

The boundary points of the arcs $t_{i}$ of the link-diagram $D_{L}$ will be
the $A$-vertices of our $m$-butterfly diagram.

Each $A$-vertex belongs to the boundary of two regions. The vertices denoted
by $B$ (and selected before) in these two regions will be called
the\textit{\ neighboring }$B$\textit{\'{}s} \textit{of the }$A$\textit{-vertex}. (In Fig. \ref{fig6}, the neighboring
$B$\'{}s of the $A$-vertex $A_{1}$ are $B_{1}$ and $B_{4}$.)

The diagram $D_{L}$ contains also \textit{crossings}. A crossing involves an
overarc and two adjacent \textit{arcs. }

We now proceed to construct an $m$-butterfly diagram $\left(  R,T\right)  .$
Joint every $A$-vertex of $D_{L}$ with its two neighboring $B$\'{}s 
by arcs lying in the regions to which these two belong. Thus we obtain a set
of arcs $R$ and we assume that these arcs have mutually disjoint interiors
among themselves and with the arcs of $T$.%

\begin{figure}
[h]
\begin{center}
\includegraphics[
height=2.0484in,
width=3.5268in
]
{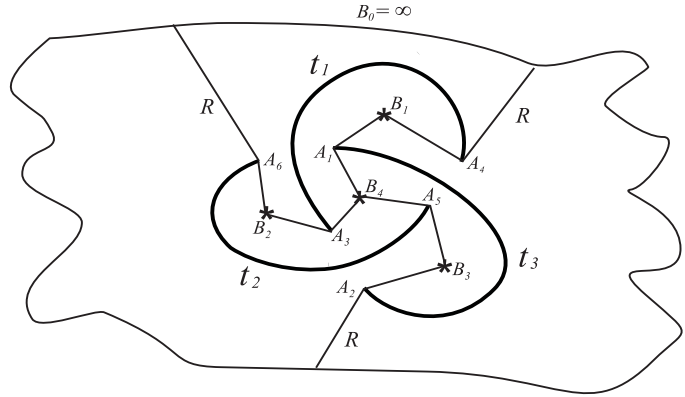}%
\caption{$m$-butterfly from a link-diagram $D_{L}.$}
\label{fig6}
\end{center}
\end{figure}

Then $\left(  R,T\right)  $ is an $m$-butterfly diagram, where $m$ is the
number of arcs in $T$. The graph $R$ is connected because the diagram $D_{L}$
is connected. Moreover, $S^{2}\backslash R$ is a disjoint union of open
$2$-cells, namely, open neighbourhoods of the arcs $t_{i}$ of the diagram.
Finally the $A$-vertices are bivalent in $R.$ Note that there are no
$E$-vertices in $R$. The set of $B$-vertices of the $m$-butterfly diagram is
the set of $B$\'{}s.

Applying the butterfly-link algorithm found in the proof of Theorem
\ref{teobridge} to $\left(  R,T\right)  $ (here the graph $\Gamma$ is the set
of dotted lines), it is easy to see that $L=L(R,T)$, see Fig. \ref{fig7}.
\begin{figure}
[h]
\begin{center}
\includegraphics[
height=2.0315in,
width=3.4844in
]%
{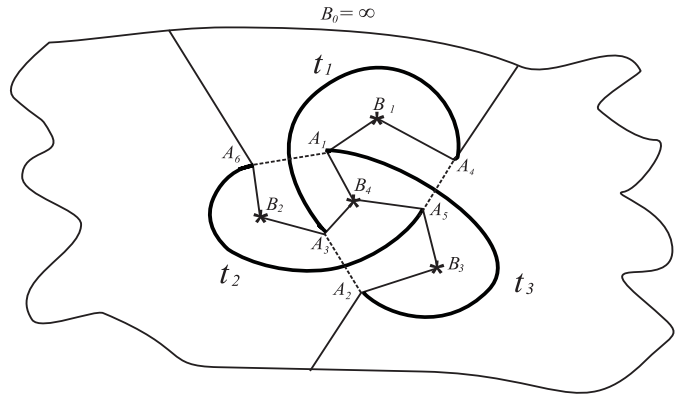}%
\caption{Link $L\left(  R,T\right)  $ from an $m$-butterfly diagram.}%
\label{fig7}
\end{center}
\end{figure}

\end{proof}

We will refer to the algorithm described in the proof of Theorem \ref{linkbut}
as the \textit{link-butterfly algorithm.}

\begin{definition}
The minimum $m$ among all possible $m$-butterfly diagrams of a given link $L$
is called the\emph{\ butterfly number} of $L$ and it is denoted by $m(L).$
\end{definition}

For example, the butterfly number of the trivial knot is 1, see Fig.
\ref{fig2} c; the butterfly number of any rational knot is 2, see the
Introduction and Fig. \ref{fig16}; and the butterfly number of the borromean
rings is 3, see Fig. \ref{fig21}.

\section{\label{Move}Trunk-reducing Move}

Our goal in the next two sections is to prove that the butterfly and bridge
number of knots and links coincide. To achieve this we need to know how to
reduce the number of trunks obtained by the link-butterfly algorithm described
in Section \ref{LinkBut}.

Let $L$ be a link and $(R,T)$ be an $m$-butterfly diagram of $L$ found by the
link-butterfly algorithm. We observed that it does not produce $E$-vertices.
Actually it produces only two types of butterflies. The butterflies, coming
from trunks that are overarcs, have more than two $A$-vertices, as illustrated
in Fig. \ref{fig8}a. The butterflies coming from trunks that are not overarcs
(simple arcs) have only two $A$-vertices. We call this last kind of
butterflies \textit{simple butterflies. }They have the shape illustrated in
Fig. \ref{fig8}b.%
\begin{figure}
[h]
\begin{center}
\includegraphics[
height=1.2785in,
width=4.1594in
]%
{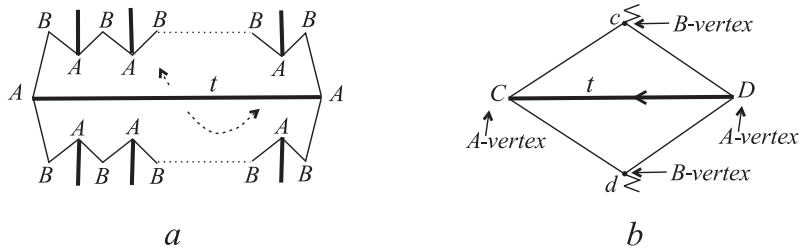}%
\caption{a. A non simple butterfly. \ \ \ \ \ \ \ \ \ b. A simple butterfly.
\ \ \ }
\label{fig8}
\end{center}
\end{figure}

We also notice that the value of $m$ in the $m$-butterfly diagram $(R,T)$ is
just the number of all arcs in the chosen link diagram.

So given a connected $m$-bridge diagram of a link $L$, together with the
$m$-butterfly diagram $(R,T)$ representation of $L$ produced using the
link-butterfly algorithm, a natural question arises:

Is it possible to make some \textit{moves} on the $m$-butterfly diagram
$(R,T),$ in such a way, that we find a different $l$-butterfly diagram
$(R^{\prime},T^{\prime})$ of $L$ but with $l<m$? We will see that we can do
this, but at the expense of producing $E$-vertices.

Now we will show how to decrease the number of butterflies in a given
$m$-butterfly. More specifically, trunks of simple butterflies will be
converted into $E$-vertices.

Let $P$ be the simple butterfly of $(R,T)$ shown in Fig. \ref{fig9}, where the
vertex labeled by $D$ at the rightmost part of the Figure is an $A$- or
$E$-vertex and the vertex labeled by $C$ at the leftmost part of the Figure is
an $A$-vertex.

\begin{figure}
[ptbh]
\begin{center}
\includegraphics[
height=1.192in,
width=3.636in
]
{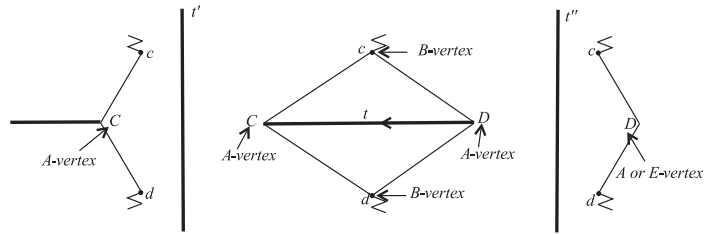}%
\caption{Simple butterfly}%
\label{fig9}
\end{center}
\end{figure}

For simplicity, we will assume here that the closed $3$-cell of $(R,T)$ is
below the paper. Consider the notations given in Fig. \ref{fig9}. On both
sides of the trunk $t^{\prime}$ we draw the arcs $C^{\prime}c$ and $C^{\prime
}d$ (See Fig. \ref{fig10}). We use the same notation on both sides, to
indicate that they match by the "reflection" along $t^{\prime}.$ Inside the
$3$-cell we trace an arc $C^{\prime}D$ getting two triangles $C^{\prime}cD$
and $C^{\prime}dD$ that have only two edges on the boundary of $(R,T).$ These
triangles together with the wings $CcD$ and $CDd$ of the simple buttterfly on
the boundary of $\partial\mathbf{B}$ can be considered as the boundary of a
pyramid with quadrilateral base $CcC^{\prime}d$ and apex $D.$%
\begin{figure}
[h]
\begin{center}
\includegraphics[
height=1.4631in,
width=4.0308in
]
{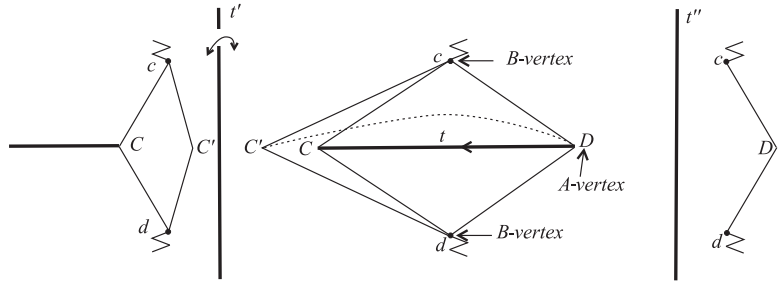}
\caption{First step}
\label{fig10}
\end{center}
\end{figure}

Now we cut the pyramid $CcC^{\prime}dD$ out of the ball $(R,T)$ (Fig.
\ref{fig11}) and glue it on the other side of $t^{\prime}$ to the
corresponding base $CcC^{\prime}d,$ thus obtaining finally Fig. \ref{fig12}.%

\begin{figure}
[ptb]
\begin{center}
\includegraphics[
height=1.4631in,
width=4.0427in
]
{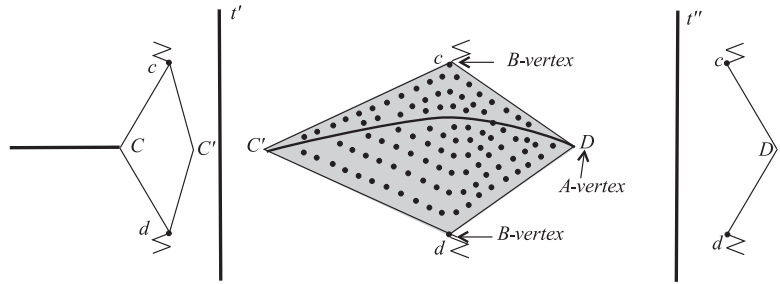}
\caption{Cutting off $CcC^{\prime}dD$}%
\label{fig11}
\end{center}
\end{figure}

\begin{figure}
[h]
\begin{center}
\includegraphics[
height=1.4631in,
width=4.0291in
]
{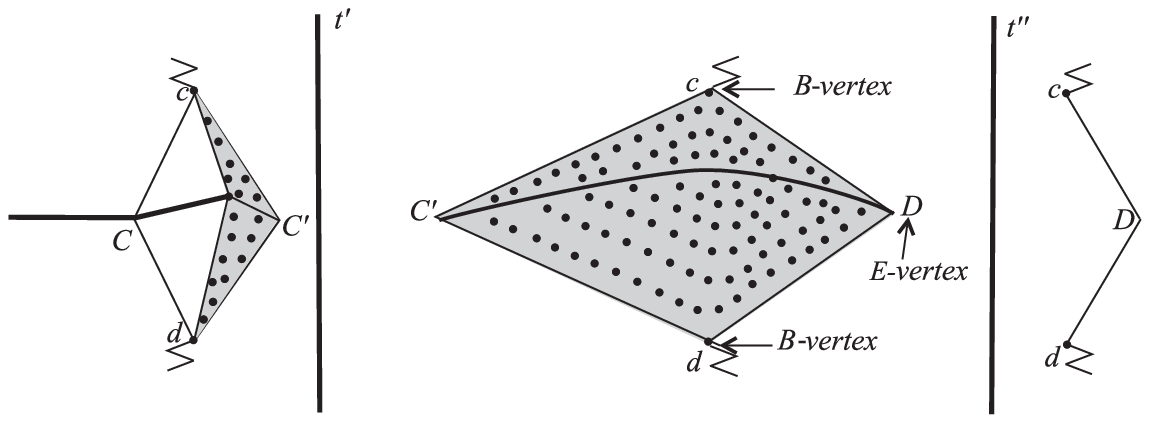}%
\caption{Gluing $CcC^{\prime}dD$}%
\label{fig12}
\end{center}
\end{figure}

\begin{figure}
[h]
\begin{center}
\includegraphics[
height=1.6859in,
width=3.2006in
]
{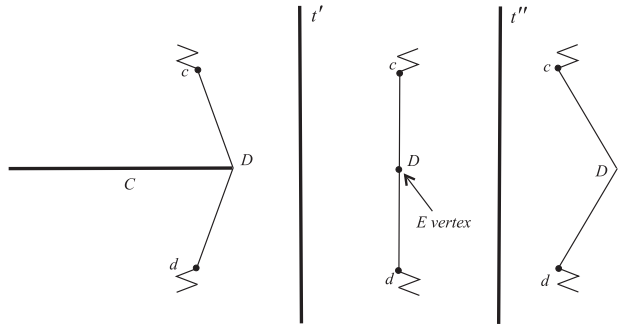}%
\caption{A new $E$-vertex}%
\label{fig13}
\end{center}
\end{figure}

In this way the simple butterfly has been substituted by two edges $cD$ and
$dD$ and a $E$-vertex $D$ (see Fig. \ref{fig13}). In this process the graph
$R$ becomes a connected graph $R_{1}$ such that $S^{2}\backslash R_{1}%
=S^{2}\backslash\left(  R\cup\bar{P}\right)  ,$ where $P$ is the simple
butterfly of $(R,T)$ shown in Fig. \ref{fig9}. Hence $S^{2}\backslash R_{1}$
consists of a disjoint union of open $2$-cells. Therefore $R_{1}$ together
with the new collection of trunks $T_{1}$ is in fact a butterfly diagram.
Moreover, note that the new $E$-vertex $D$ is bivalent in $R$. See the center
part of Fig. \ref{fig13}. The point $C$ is not any more an $A$-vertex
(actually, it is now a point in the interior of a trunk$,$ (See the leftmost
part of Fig. \ref{fig13}), and notice that the valence of the $B$-vertices of
the simple butterfly $P$ decreases by one. Recall that a vertex of $R$ is a
$B$-vertex iff $p^{-1}\left(  p\left(  v\right)  \right)  $ contains at least
one non-bivalent vertex, where $p:\mathbf{B}^{3}\rightarrow M(R,T)$ is the
identification map. So, it is possible that some of the $B$-vertices are not
any more $B$-vertices but it is not a problem since they can be considered as
any other point in $R_{1}$ that is not a vertex.

The transition from Fig. \ref{fig9} to Fig. \ref{fig13} will be referred to as
a \textit{\textquotedblleft trunk-reducing move\textquotedblright.}

We have proved the following theorem

\begin{theorem}
A trunk-reducing move converts an $m$-butterfly diagram of a link $L$ into an
$\left(  m-1\right)  $-butterfly diagram of the same link $L$. The new diagram
gets a new $E$-vertex in place of a simple butterfly.
\end{theorem}

\begin{proof}
It is enough to apply the butterfly-link algorithm to both butterfly diagrams.
Apply it to Figures \ref{fig9} and \ref{fig13}.
\end{proof}

\begin{example}
\label{EjemploTrebol}Let us apply trunk-reducing moves to the $4$-butterfly
diagram of the trefoil knot illustrated in Fig. \ref{fig14}. There, we have
four trunks: $t_{1},t_{2},t_{3},t_{4},$ and six $B$-vertices; $a,b,c,d,e,f$
corresponding to each region of the diagram of the knot. For simplicity, we do
not draw the edges joining $A$- and $B$-vertices of the corresponding
butterfly $(R,T).$

\begin{figure}
[h]
\begin{center}
\includegraphics[
height=1.9527in,
width=1.4707in
]
{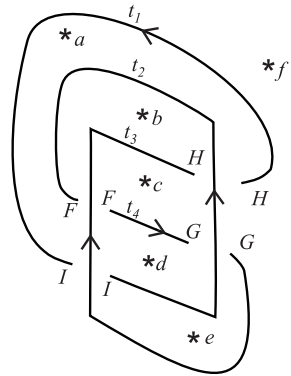}
\caption{A 4-butterfly representation of the trefoil knot.}
\label{fig14}
\end{center}
\end{figure}

The arcs $t_{1}$ and $t_{4}$ correspond to simple butterflies. Therefore,
performing two trunk-reducing moves in $t_{1}$ and $t_{4}$ (in this order)$,$
the trunks $t_{1}$ and $t_{4}$ are reduced to the $E$-vertices labeled by
$E_{1}$ and $E_{4}$, respectively (see Fig. \ref{fig15}). The diagram of Fig.
\ref{fig15} is not yet a butterfly diagram because it contains too many
vertices. Indeed, under the application of the trunk-reducing moves the
$B$-vertices of the original diagram become bivalent vertices of the new
diagram that are not $A$-vertices nor $E$-vertices. Therefore we can delete
them, thus obtaining the $2$-butterfly diagram of Fig. \ref{fig16}.
\end{example}

\begin{figure}
[ptb]
\begin{center}
\includegraphics[
height=1.9654in,
width=1.4207in
]
{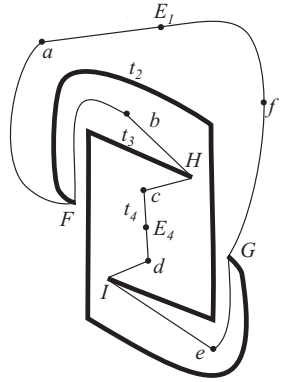}
\caption{Diagram with new $E$-vertices. }
\label{fig15}
\end{center}
\end{figure}

\begin{figure}
[ptb]
\begin{center}
\includegraphics[
height=1.8239in,
width=2.9659in
]
{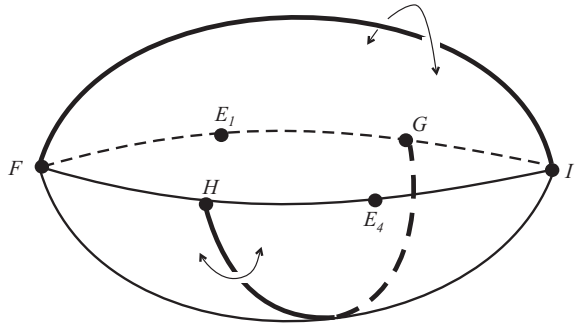}%
\caption{A 2-butterfly representation of the trefoil knot.}%
\label{fig16}
\end{center}
\end{figure}

A $4$-butterfly diagram for the trivial link with two components is depicted
in Fig. \ref{fig17}.

\begin{figure}
[h]
\begin{center}
\includegraphics[
height=2.6228in,
width=3.0032in
]
{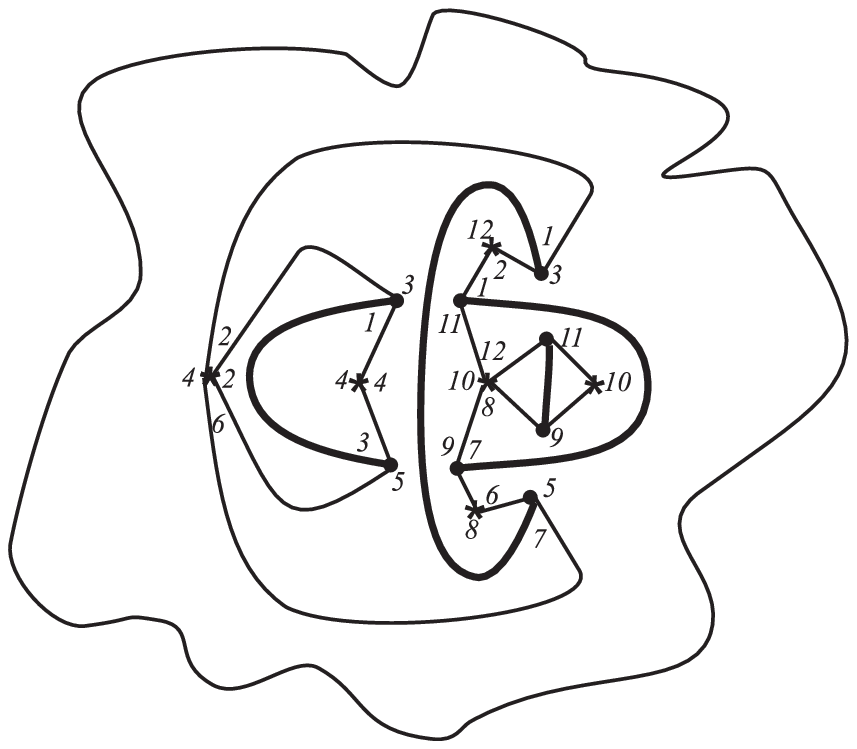}%
\caption{A $4$-butterfly for the trivial link with two components.}%
\label{fig17}
\end{center}
\end{figure}

Applying one trunk-reducing move we get Fig. \ref{fig1}.
\begin{figure}
[h]
\begin{center}
\includegraphics[
height=1.9459in,
width=2.4517in
]
{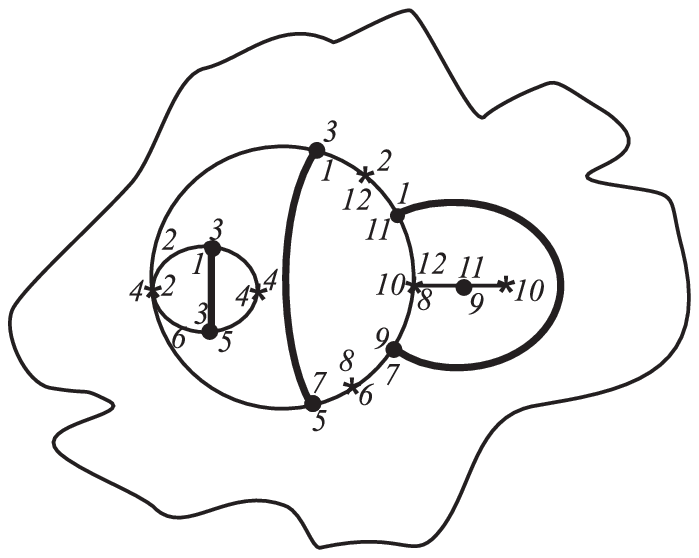}
\caption{A $3$-butterfly for the trivial link with two components, obtained by
a trunk-reducing move.}
\label{fig1}
\end{center}
\end{figure}
%EndExpansion
A second trunk-reducing move produces the $2$-butterfly diagram representing
the trivial link with two components shown in Fig. \ref{fig18}a. In Fig.
\ref{fig18}b we apply the butterfly-link algorithm to the $2$-butterfly to
recover the link.

\begin{figure}
[ptbh]
\begin{center}
\includegraphics[
height=1.7689in,
width=3.6513in
]
{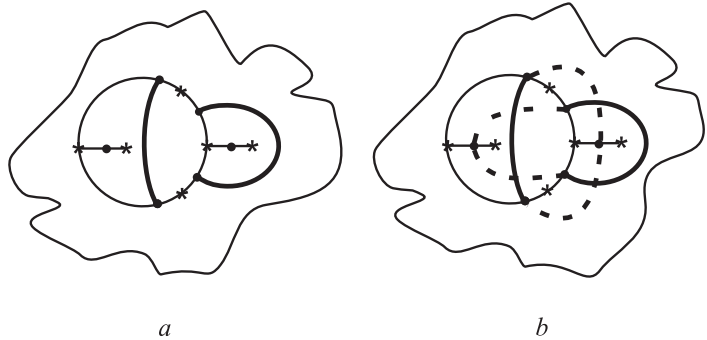}%
\caption{ A $2$-butterfly diagram representing the trivial link with two
components.}
\label{fig18}
\end{center}
\end{figure}

\begin{remark}
\label{inverso} The inverse of a trunk-reducing move can certainly be applied
to any $E$-vertex in an $m$-butterfly diagram to increase the number of
trunks. In this way it is always possible to obtain a butterfly diagram
without $E$-vertices from any given butterfly diagram of a link.
\end{remark}

\section{\label{number}The Bridge Number and the Butterfly Number}

Let us remark that the knot-diagram of the trefoil knot given in Example
\ref{EjemploTrebol} corresponds to a $2$-bridge presentation of it and by
applying trunk-reducing moves we obtained a $2$-butterfly diagram of the
trefoil knot. Actually this is a general result, and we want to show that for
any link $L,$ the butterfly number equals the bridge number, i.e., $m(L)=b(L)$.

\begin{theorem}
\label{TeoremaBridgeigualmariposa}For any link $L,$ $b(L)=m(L).$
\end{theorem}

\begin{proof}
The fact that $b(L)\leq m(L)$ is a corollary of Theorem \ref{teobridge}.

Now we will show that $m(L)\leq b(L)$ for any link $L.$

Let $D_{L}$ be a link-diagram of $L$, such that it satisfies the conditions of
Lemma \ref{diagram} and\ the number of bridges (or overarcs) is $b(L)$.

We can apply the link-butterfly algorithm to $D_{L}$ to obtain an
$m$-butterfly diagram $\left(  R,T\right)  $ without $E$-vertices, where $m$
is the number of arcs of $D_{L}$ (Theorem \ref{linkbut}).

Next apply trunk-reducing moves to $\left(  R,T\right)  $ in order to trade
simple butterflies by pairs of edges and $E$-vertices. We have to be careful
because we cannot apply the trunk-reducing moves at random. (Remember that to
be able to apply a trunk-reducing move we need that one of the two
neighbouring vertices be an $A$-vertex.) To have a consistent order of
application for a component $L_{i}$ of $L$, we start with an overarc of the
projection of $L_{i}$ (granted by Proposition \ref{diagram}) and we tour
$L_{i},$ following some orientation, performing trunk-reducing moves to the
simple butterflies in the same order that they are found. In this way we
eliminate all the simple arcs belonging to $L_{i}$ and convert them into
$E$-points. We do this for every component of $L.$ Therefore all simple
butterflies disappear (converted into $E$-vertices) and there remains only the
trunks coming from overarcs. Since the number of overarcs of $D_{L}$ is $b(L)$
the new butterfly is a $b(L)$-butterfly diagram. Then $m(L)\leq b(L).$
\end{proof}

\begin{example}
Consider the 3-bridge presentation of the borromean rings given in Fig.
\ref{fig20}.
\begin{figure}
[ptbh]
\begin{center}
\includegraphics[
height=2.4251in,
width=2.4408in
]
{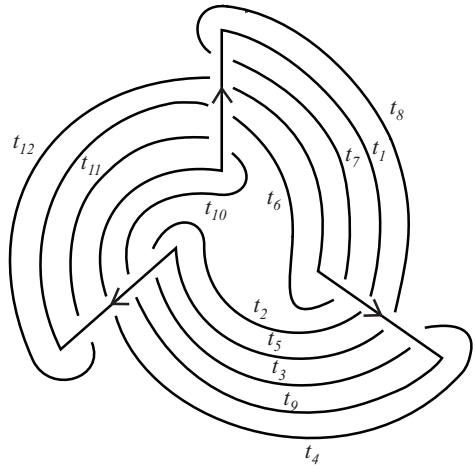}
\caption{ A 3-bridge presentation.}
\label{fig20}
\end{center}
\end{figure}
%EndExpansion
Make trunk-reducing moves first to the sequence $t_{2},t_{3},t_{4}$. Next to
the sequence $t_{6},t_{7,}t_{8},$ and finally to the sequence $t_{10}%
,t_{11},t_{12}.$ You will get the 3-butterfly diagram of Fig. \ref{fig21},
where those trunks have been exchanged by the $E$-vertices $A,B,C,D,E,F,G,H,$
and $I$, respectively. The vertices $o$, $\infty$, $1$, $2 $, $3$, $4$, $5$,
$6$, $7$, $8$, $9$,$10$, $11$, and $12$ are $B$-vertices and all of them
belong to the orbit of $\left\{  o\right\}  $, under the equivalence relation
$\simeq.$

\begin{figure}
[ptbh]
\begin{center}
\includegraphics[
height=2.411in,
width=2.6482in
]
{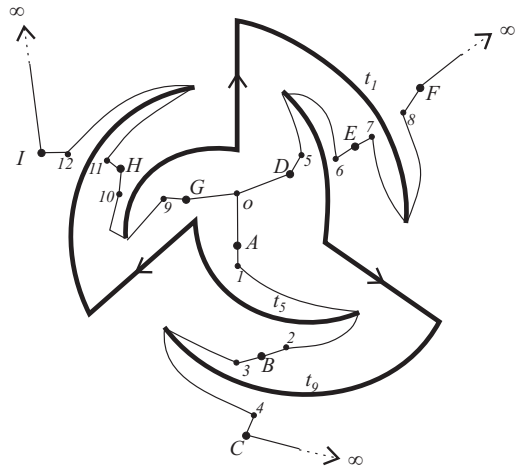}
\caption{ A 3-butterfly.}
\label{fig21}
\end{center}
\end{figure}

Another way to visualize this 3-butterfly diagram is shown in Fig.
\ref{fig22}, where, for simplicity, we do not mark the $B$-vertices, except
$o$ and $\infty.$
\end{example}

\begin{figure}
[ptbh]
\begin{center}
\includegraphics[
height=2.399in,
width=2.2485in
]
{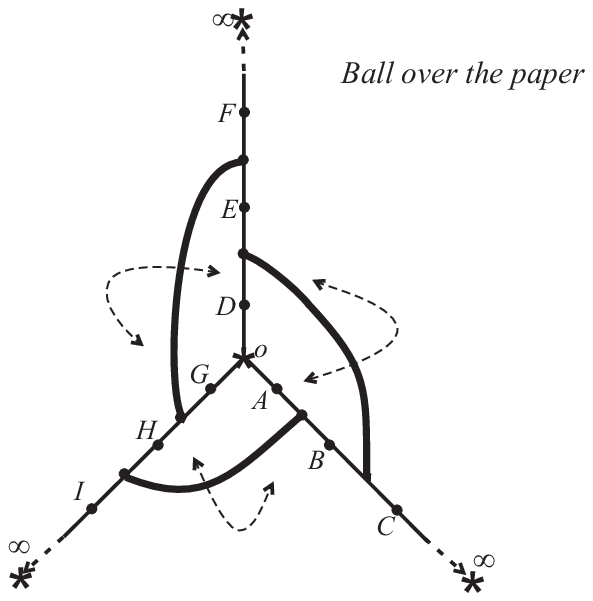}%
\caption{A 3-butterfly diagram for de borromean rings (without some
B-vertices).}
\label{fig22}
\end{center}
\end{figure}

\section{Conclusions}

We have proved that any link can be represented as an $m$-butterfly. We
defined the butterfly number of a link and we proved that the butterfly number
equals the bridge number of a link. Therefore it is feasible to study the
$m$-bridge links via $m$-butterflies. For each $2$-bridge link the associated
$2$-butterfly allow us to visualize the corresponding rational number. For
example, in Fig. \ref{fig2}a we have a 2-butterfly that represents the
rational knot $5/2.$ For the 3-bridge links, as we have announced in the
introduction, it is possible to associate a set of 3 rational numbers to each
3-butterfly. For more details about the way to assign a set of three rational
numbers to a $3$-butterfly diagram see \cite{To}, \cite{HMTT4}.

For example, in Fig. \ref{fig23} we show the diagrams of two 3-butterflies,
$(R_{1},T_{1})$ and $(R_{2},T_{2}),$ with the associated set of rational
numbers.
\begin{figure}
[ptb]
\begin{center}
\includegraphics[
height=2.3903in,
width=3.9729in
]%
{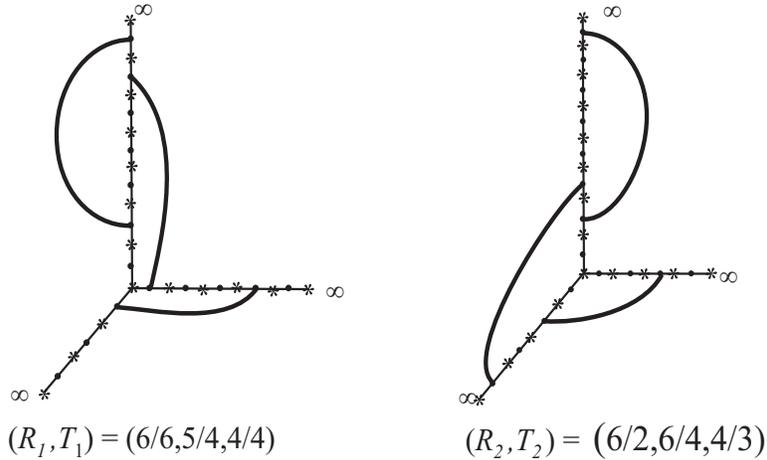}
\caption{Two butterfly diagrams for the knot $8_{20}$, with the associate
rational numbers.}
\label{fig23}%
\end{center}
\end{figure}

The two diagrams are different, however $L\left(  R_{1},T_{1}\right)  $ and
$L\left(  R_{2},T_{2}\right)  $ are equivalent $3$-bridge presentation of the
knot $8_{20}$ with bridge (and butterfly) number $3$. To exhibit the
equivalence between the bridge presentations $L\left(  R_{1},T_{1}\right)  $
and $L\left(  R_{2},T_{2}\right)  $ we modify the presentation of the two
$3$-butterfly diagrams $(R_{1},T_{1})$ and $(R_{2},T_{2})$, as shown in Fig.
\ref{fig28}, on the left. In the center we have the link diagrams obtained
when we close the $3$-butterflies.

\begin{figure}
[ptb]
\begin{center}
\includegraphics[
height=3.426in,
width=4.545in
]
{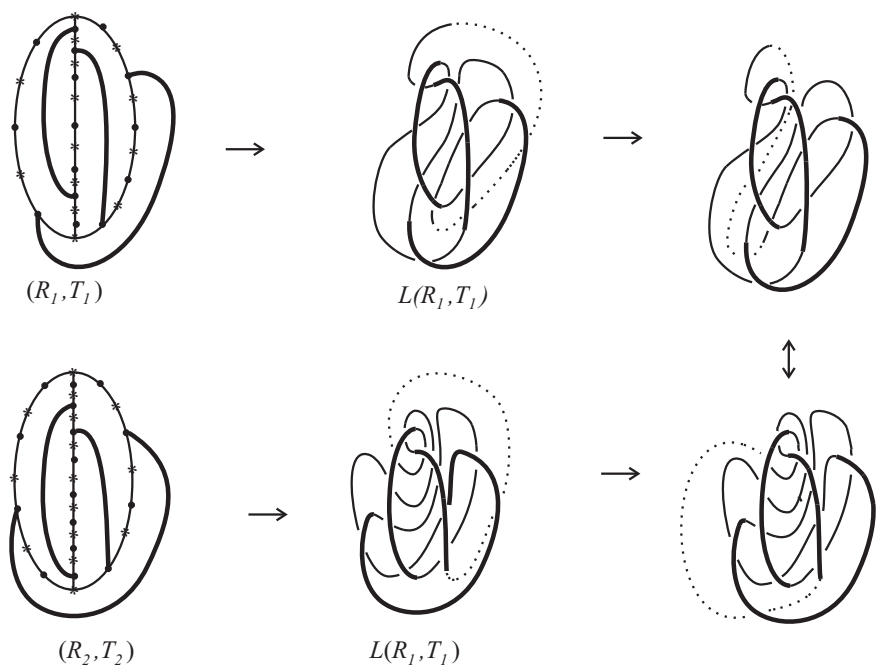}%
\caption{Two non equivalent 3-butterfly diagram with equivalent 3-bridge
presentation.}%
\label{fig28}%
\end{center}
\end{figure}

Then we move the dotted arc as shown in each diagram.

This raises the problem \textit{of finding a set of moves} in a butterfly
diagram connecting diagrams representing the same link. This is left as an
open problem.

\end{document}